\documentclass[11pt, a4paper, oneside]{amsart}

\usepackage[english]{babel}
\usepackage{amsmath, amsthm, amsfonts, mathrsfs, amssymb}
\usepackage{mathtools}
\mathtoolsset{centercolon}
\usepackage{booktabs}
\usepackage[shortlabels]{enumitem}
\setlist[itemize]{leftmargin=20pt}
\usepackage[colorlinks]{hyperref}

\usepackage{color}
\usepackage{graphicx}

\newcommand{\N}{\ensuremath{\mathbb{N}}}

\newcommand{\R}{\ensuremath{\mathbb{R}}}

\newcommand{\mc}{\mathcal}

\DeclarePairedDelimiter\abs{\lvert}{\rvert}

\DeclarePairedDelimiter\cbrace\{\}
\DeclarePairedDelimiter\ha()
\DeclarePairedDelimiter{\ip}\langle\rangle
\DeclarePairedDelimiter{\nrm}\lVert\rVert

\newcommand{\nrmb}[1]{\bigl\|#1\bigr\|}

\newcommand{\hab}[1]{\bigl(#1\bigr)}
\newcommand{\cbraceb}[1]{\bigl\{#1\bigr\}}
\newcommand{\ipb}[1]{\bigl\langle#1\bigr\rangle}

\newcommand{\nrms}[1]{\Bigl\|#1\Bigr\|}

\newcommand{\has}[1]{\Bigl(#1\Bigr)}
\newcommand{\cbraces}[1]{\Bigl\{#1\Bigr\}}

\DeclareMathOperator{\ind}{\mathbf{1}}
\DeclareMathOperator{\UMD}{UMD}

\DeclareMathOperator{\ch}{ch}
\DeclareMathOperator{\parent}{\pi}

\newcommand{\dd}{\hspace{2pt}\mathrm{d}}
\newcommand{\dx}{\mathrm{d}x}

\newcommand{\comp}{\mathsf{c}}

\newcommand{\cdotroomy}{\,\cdot\,}

\renewcommand{\l}{\ensuremath{\ell}}

\newtheorem{theorem}{Theorem}
\newtheorem{corollary}[theorem]{Corollary}
\newtheorem{lemma}[theorem]{Lemma}
\newtheorem{proposition}[theorem]{Proposition}
\newtheorem{problem}[theorem]{Problem}
\newtheorem{innercustomthm}{Theorem}
\newenvironment{theoremnumber}[1]
{\renewcommand\theinnercustomthm{#1}\innercustomthm}{\endinnercustomthm}

\theoremstyle{remark}
\newtheorem*{remark}{Remark}
\newtheorem*{example}{Example}

\theoremstyle{definition}
\newtheorem{definition}[theorem]{Definition}

\numberwithin{theorem}{section}
\numberwithin{equation}{section}

\title[Sparse domination for the lattice maximal operator]{Sparse domination for the lattice Hardy--Littlewood maximal operator}
\author{Timo S. H\"anninen and Emiel Lorist}
\thanks{T.S.H. is supported by the Academy of Finland (Funding Decision No 297929). He is a member of the Finnish Centre of Excellence in Analysis and Dynamics Research. E.L. is supported by the VIDI subsidy 639.032.427 of the Netherlands Organisation for Scientific Research (NWO)}
\address{Department of Mathematics and Statistics \\ University of Helsinki \\ P.O. Box 68 \\FI-00014 Helsinki \\ Finland}
\email{timo.s.hanninen@helsinki.fi}
\address{Delft Institute of Applied Mathematics \\ Delft University of Technology \\ P.O. Box 5031\\ 2600 GA Delft \\The Netherlands}
\email{e.lorist@tudelft.nl}

\allowdisplaybreaks
\begin{document}
  \begin{abstract}
  We study the domination of the lattice Hardy--Littlewood maximal operator by sparse operators in the setting of general Banach lattices. We prove that the admissible exponents of the dominating sparse operator are determined by the $q$-convexity of the Banach lattice.
\end{abstract}




\maketitle

\section{Introduction}Various complicated operators of harmonic analysis can be dominated by simple {\it sparse operators} and, via such domination, weighted estimates for them follow from estimates for sparse operators. This approach, in its essence, was initiated by Lerner by his median oscillation decomposition \cite{Le10}. Its early highlight was the domination of Calder\'{o}n--Zygmund operators by sparse operators by Lerner \cite{Le13b,Le13a}. This domination yielded an alternative, simple proof of the $A_2$ theorem, which was originally proved by Hyt\"onen \cite{Hy12}. Since then, a wide variety of operators has been dominated by sparse operators (or, more generally, {\it sparse forms}). We refer the reader to the introductions, for example, in \cite{BFP16,CCDO17, Le16, LMR17, LN15} for an overview of this vast field.

We study pointwise domination of the {\it lattice Hardy--Littlewood maximal operator} by {\it sparse operators} in the setting of general Banach lattices. Let $E$ be a Banach lattice, that is, a Banach space and a lattice such that both the structures are compatible. An important class of Banach lattices is the class of Banach function spaces, with the pointwise order as the lattice partial order.  For an introduction to Banach lattices, see for example \cite{LT79,MN91}. Let $\mu$ be a locally finite Borel measure on $\R^d$, and $\mc{D}$ be a finite collection of dyadic cubes in $\R^d$. A subcollection $\mc{S}\subseteq\mc{D} $ of dyadic cubes is called {\it sparse} if for every $S\in \mc{S}$ there exists a subset $E_S\subseteq S$ such that $\mu(E_S)\geq \frac12\mu(S)$ and such that the sets $\{E_S\}_{S\in\mc{S}}$ are pairwise disjoint. The operators of study are defined as follows:
\begin{itemize}
\item The {\it dyadic lattice Hardy--Littlewood maximal operator} $\widetilde{M}^\mu_{\mc{D}}$ is defined as follows: For a locally integrable function $f\colon \R^d \to E$, we set
\begin{equation*}
\widetilde{M}^\mu_{\mc{D}}f(x):=\sup_{Q\in\mc{D}} \ip{\abs{f}}_Q^\mu \ind_Q(x),\qquad x \in \R^d,
\end{equation*}
where the supremum and the absolute value are taken in the lattice sense, and $\ip{f}_Q^\mu:=\frac{1}{\mu(Q)} \int_Q f \dd \mu$.
\item For an exponent $q \in [1,\infty)$ and a sparse collection  $\mc{S}$ of dyadic cubes, the {\it sparse operator} $\mc{A}^\mu_{q,\mc{S}}$ relevant to our study is defined as follows: For a locally integrable function $f:\R^d \to \mathbb{R}$, we set
\begin{equation*}
\mc{A}^\mu_{q,\mc{S}}f(x):=\big(\sum_{S\in\mc{S}} \hab{\ip{\abs{f}}^\mu_S}^q \ind_S(x) \big)^{1/q}, \qquad x \in \R^d.
\end{equation*}
\end{itemize}

We address the following problem:
\begin{problem}\label{prob:main}For a Banach lattice $E$, for which exponents $q\in[1,\infty)$ can the dyadic  lattice maximal operator $\widetilde{M}^\mu_{\mc{D}}$ be pointwise dominated by a sparse operator $\mc{A}^\mu_{q,\mc{S}}$?
\end{problem}
The pointwise domination is meant in the following sense:
 For each locally integrable function $f\colon \R^d \to E$ and for each finite collection $\mc{D}$ of dyadic cubes there exists a sparse subcollection $\mc{S}\subseteq \mc{D}$ such that
\begin{equation}
\label{eq:domination}
\nrmb{\widetilde{M}^\mu_{\mc{D}}f}_E \leq C_{E,q} \, \mc{A}^\mu_{q,\mc{S}}(\nrm{f}_E) \quad\text{$\mu$-a.e.}
\end{equation}
Note that the larger the exponent $q$, the smaller the dominating sparse operator $\mc{A}^\mu_{q,\mc{S}}$ and hence the problem is to determine the largest possible exponent in the domination.

We study the problem among the Banach lattices $E$ that have the Hardy--Littlewood property. In the most important case that the measure $\mu$ is the Lebesgue measure, this assumption is necessary for the domination, for any domination exponent $q\in[1,\infty)$. The class of Banach lattices with the Hardy--Littlewood property includes all reflexive Lebesgue, Lorentz, and Orlicz spaces. The Hardy--Littlewood property is defined and further discussed in Section \ref{section:preliminaries}.

We find out that the admissible exponents are determined by the geometric property of the lattice $E$ called $q$-convexity. We recall that a Banach lattice $E$ is called {\it $q$-convex}, with $q \in [1,\infty)$, if
$$
\nrms{\big(\sum_{k=1}^n \abs{e_k}^q\big)^{1/q}}_E\leq C_{E,q} \sum_{k=1}^n \big(\nrm{e_k}_E^q\big)^{1/q}
$$
for all $e_1,\cdots,e_n \in E$. More precisely, we show that the exponent $q^*$, defined by $$q^*:=\sup \{ q\in(1,\infty): \text{$E$ is $q$-convex}\},$$ is critical in that the domination \eqref{eq:domination} holds for all $q\in[1,q^*)$ and fails for all $q\in(q^*,\infty)$.

We first study the necessity of $q$-convexity. The main contribution of this article reads as follows:

\begin{theorem}\label{theorem:sharpness}Let $E$ be a Banach lattice, let $\mu$ be a locally finite Borel measure such that $\mu(\R^d)= \infty$ and let $r \in (1,\infty)$. Assume that for each finite collection $\mc{D}$ of dyadic cubes and for each locally integrable function $f:\R^d\to E$ there exists a sparse collection $\mc{S}\subseteq \mc{D}$ such that
$$
\nrmb{\widetilde{M}^\mu_{\mc{D}}f(x)}_E \leq C_{E,q} \,\mc{A}^\mu_{r,\mc{S}}(\nrm{f}_E)(x), \qquad \text{$\mu$-a.e. $x\in \R^d$.}
$$
Then the Banach lattice $E$ is $q$-convex for all exponents $q\in[1,r)$.
\end{theorem}

We then study the sufficiency of $q$-convexity. For the particular Banach lattice $E=\ell^q$, a prototypical example of a $q$-convex lattice, the domination  was obtained by Cruz-Uribe, Martell, and P\'{e}rez  \cite[Section 8]{CMP12}. In this article, we mention how their proof, based on Lerner's median oscillation decomposition, can be extended to general  Banach lattices $E$. We also give an alternative, elementary proof of this domination, via the technique of stopping cubes. In this proof, the lattice-valued setting differs from the scalar-valued setting in that we need to use a lattice-valued generalization of the usual Muckenhoupt--Wheeden {\it principal cubes} stopping condition. The domination in full generality reads as follows:

\begin{theorem}\label{theorem:domination}Let $E$ be a Banach lattice and let $\mu$ be a locally finite Borel measure. Assume that $E$ has the Hardy--Littlewood property and is $q$-convex for some $q\in(1,\infty)$. Then for each finite collection $\mc{D}$ of dyadic cubes and for each locally integrable function $f:\R^d\to E$ there exists a sparse collection $\mc{S}\subseteq \mc{D}$ such that
\begin{equation*}
\nrmb{\widetilde{M}^\mu_{\mc{D}}f(x)}_E \leq C_{E,q} \,\mc{A}^\mu_{q,\mc{S}}(\nrm{f}_E)(x), \qquad \text{$\mu$-a.e. $x\in \R^d$.}
\end{equation*}
\end{theorem}
As an immediate corollary of the domination, we mention sharp weighted weak and strong $L^p$-estimates for the non-dyadic lattice Hardy--Littlewood maximal operator (see Corollary \ref{corollary:weightedestimate}).

Combining Theorem \ref{theorem:domination} and Theorem \ref{theorem:sharpness} yields the following corollary, which has been mentioned above:
\begin{corollary}[Admissible exponents are determined by $q$-convexity]\label{corollary:sharpness}Let $E$ be a Banach lattice with the Hardy--Littlewood property and let $\mu$ be a locally finite Borel measure such that $\mu(\R^d)= \infty$.  Define
$$q^*:=\sup\{q\in(1,\infty) : \text{$E$ is $q$-convex}\}.$$
Then the sparse domination \eqref{eq:domination} holds for all $q\in[1,q^*)$ and fails for all $q\in(q^*,\infty)$.
\end{corollary}
\begin{remark}
A Banach lattice $E$ may be $q^*$-convex (for example, $\ell^q$) or may fail to be $q^*$-convex (for example, $L^{p,q}$ with $p <q$). If $E$ is $q^*$-convex, then the sparse domination \eqref{eq:domination} holds for $q^*$ by Theorem \ref{theorem:domination}. We do not know whether the converse of this holds: is it true that if the sparse domination \eqref{eq:domination} holds for $q^*$, then $E$ is $q^*$-convex; or in other words, is it true that if $E$ is not $q^*$-convex, then the sparse domination \eqref{eq:domination} fails for $q^*$?
\end{remark}

This article is organized as follows: We summarize the preliminaries in Section \ref{section:preliminaries}. We then prove that the domination implies the $q$-convexity (Theorem \ref{theorem:sharpness}) in Section \ref{section:sharpness}. Furthermore, we give an alternative proof of the result that the domination is implied by the $q$-convexity (Theorem \ref{theorem:domination}) in Section \ref{sec:sparse}, and mention weighted bounds (Corollary \ref{corollary:weightedestimate}) as its corollary in Section \ref{sec:weighted_estimates}. In Appendix \ref{sec:lpl1}, for the reader's convenience, we give a self-contained elementary proof of the well-known fact that the strong $L^p$-bound with $p\in(1,\infty)$ implies the weak $L^1$-bound for the dyadic lattice maximal operator (Proposition \ref{prop:calderonzygmund}). This fact is used in our proof of the sparse domination.

\textbf{Acknowledgement.} The authors thank Mark Veraar for his helpful comments on the draft.

\section{Preliminaries}\label{section:preliminaries}
Let $\mu$ be a locally finite Borel measure on $\R^d$, and $\mc{D}$ be a finite collection of dyadic cubes in $\R^d$. It is well-known that, for every $q\in(0,\infty)$, the sparse operator $\mc{A}^\mu_{q,\mc{S}}$, defined in the introduction, is bounded on $L^p(\mu)$ for every $p\in(1,\infty)$. This can be checked, for example, by using duality and the Hardy--Littlewood maximal inequality. Therefore, a necessary condition for the domination  \eqref{eq:domination}
is that the  dyadic lattice maximal operator $\widetilde{M}^\mu_{\mc{D}}$ is bounded on $L^p(\mu;E)$. In our context the most important measure is the Lebesgue measure, which leads us to consider the Banach lattices that have the {\it Hardy--Littlewood property}:

\begin{definition}[Hardy--Littlewood property]\label{def:hardy_littlewood_property}A Banach lattice $E$ has the {\it Hardy--Littlewood property} if for some $p\in(1,\infty)$, we have
\begin{equation}\label{eq:HLproperty}
  \sup_{\text{$\mc{D}$}} \nrmb{\widetilde{M}^{\mathrm{d}x}_{\mc{D}}}_{L^p(\dx;E)\to L^p(\dx;E)}<\infty,
\end{equation}
where the supremum is taken over all finite collections $\mc{D}$ of dyadic cubes and $\mathrm{d}x$ denotes the Lebesgue measure.
\end{definition}

\begin{remark} ~
\begin{itemize}
  \item By a covering argument using shifted dyadic systems (see for example \cite[Lemma 3.2.26]{HNVW16}), it is equivalent to take the supremum in \eqref{eq:HLproperty} over all finite collections of generic cubes or balls, in place of taking it over all finite collections of dyadic cubes over several dyadic systems.
  \item The Hardy-Littlewood property is independent of the exponent $p$  and of the dimension $d$ (see \cite[Remark 1.3 and Theorem 1.7]{GMT93} or \cite[Theorem 3]{DK17b}). The independence of the exponent $p$ also follows from the sparse domination (Theorem \ref{theorem:domination}), since the dominating sparse operator is bounded on $L^p$ for all $p\in(1,\infty)$.
  \item Among all the measures on $\R^d$, the norm of the lattice maximal operator with respect to the Lebesgue measure is the largest  (see \cite[Appendix A.2]{Ha17}), in that for every locally finite Borel measure $\mu$ and for every finite collection $\mc{D}$ of dyadic cubes, we have
$$
\nrmb{\widetilde{M}^\mu_{\mc{D}}}_{L^p(\mu;E)\to L^p(\mu;E)}\lesssim  \sup_{\text{$\mc{D'}$}} \nrmb{\widetilde{M}^{\mathrm{d}x}_{\mc{D'}}}_{L^p(\dx;E)\to L^p(\dx;E)}.
$$
\end{itemize}
\end{remark}

\begin{example}~
\begin{itemize}
  \item The Fefferman--Stein vector-valued maximal inequality states that the Banach lattice $\ell^q$ with $q\in(1,\infty]$ has the Hardy--Littlewood property.
  \item Every Banach lattice with  the $\UMD$ property (Unconditional Martingale Differences) has the Hardy-Littlewood property \cite{Bo84,Ru86}. The class of Banach lattices with the $\UMD$ property and hence with the Hardy--Littlewood property includes all reflexive Lebesgue, Lorentz and Orlicz spaces. For $\UMD$ spaces, see for example \cite[Chapter 4]{HNVW16}.
\end{itemize}
\end{example}

It is known that the domination \eqref{eq:domination} holds with the exponent $q=1$. This follows from viewing the operator $\widetilde{M}^\mu_{\mc{D}}$ as an instance of a singular integral operator or a discrete analogue of such, operators for which the domination with $q=1$ is known:
 \begin{itemize}
   \item $\widetilde{M}^\mu_{\mc{D}}$ can be viewed as a vector-valued singular integral (see \cite{GMT93, GMT98}). The sparse domination for vector-valued singular integrals follows by combining \cite[Theorem 2.10]{HH14} (dominating vector-valued singular integrals by more complex operators) and \cite[Theorem A]{CR16} (dominating the more complex operators by the sparse operator $\mc{A}^\mu_{1,\mc{S}}$).
   \item $\widetilde{M}^\mu_{\mc{D}}$ can be viewed as a vector-valued martingale transform (see \cite{MT00}). Vector-valued martingale transforms can be dominated by the sparse operator $\mc{A}^\mu_{1,\mc{S}}$ (see \cite[Theorem 2.4]{La17b}; for an alternative proof, see \cite[Proposition 2.7]{Ha17b}).
 \end{itemize}
As stated in Problem \ref{prob:main}, our purpose is to study whether the domination \eqref{eq:domination} holds with some strictly larger exponent $q\in(1,\infty)$. The critical notion for this is that of {\it $q$-convexity}:
  \begin{definition}[$q$-convexity]\label{def:q_convexity}
    We say that a Banach lattice $E$ is {\it $q$-convex}, with $q \in [1,\infty)$, if
$$
\nrms{\big(\sum_{k=1}^n \abs{e_k}^q\big)^{1/q}}_E\leq C_{E,q} \sum_{k=1}^n \big(\nrm{e_k}_E^q\big)^{1/q}
$$
for all $e_1,\cdots,e_n \in E$.
 \end{definition}
 Note that the expression $\big(\sum_{k=1}^n \abs{e_k}^q\big)^{1/q}$ can be defined pointwise in a Banach function space. In a general lattice it can be defined using the Krivine calculus (see for example \cite[Theorem 1.d.1]{LT79}).

Every Banach lattice with the Hardy--Littlewood property is $q$-convex for some $q>1$ \cite[Theorem 2.8]{GMT93}. Recall that, in the case that the measure $\mu$ is the Lebesgue measure, the Hardy--Littlewood property is necessary for the domination \eqref{eq:domination} to hold for any $q\in[1,\infty)$. Thus, in the case of the Lebesgue measure, if the domination $\eqref{eq:domination}$ holds for any exponent $q\in[1,\infty)$, then the lattice $E$ is $q$-convex for some $q\in(1,\infty)$.

\section{Domination exponent is determined by \texorpdfstring{$q$}{q}-convexity}\label{section:sharpness}
In this section we prove Theorem \ref{theorem:sharpness} from the introduction, which states the necessity of the $q$-convexity assumption for the domination \eqref{eq:domination} to hold:
\begin{theoremnumber}{\ref*{theorem:sharpness}}Let $E$ be a Banach lattice, let $\mu$ be a locally finite Borel measure such that $\mu(\R^d)= \infty$ and let $r \in (1,\infty)$. Assume that for each finite collection $\mc{D}$ of dyadic cubes and for each locally integrable function $f:\R^d\to E$ there exists a sparse collection $\mc{S}\subseteq \mc{D}$ such that
$$
\nrmb{\widetilde{M}^\mu_{\mc{D}}f(x)}_E \leq C_{E,q} \,\mc{A}^\mu_{r,\mc{S}}(\nrm{f}_E)(x), \qquad \text{$\mu$-a.e. $x\in \R^d$.}
$$
Then the Banach lattice $E$ is $q$-convex for all exponents $q\in[1,r)$.
\end{theoremnumber}

\begin{proof}
Let $Q_0$ be a dyadic cube such that $\mu(Q_0)>0$ and such that for any $C>0$ there exists a dyadic cube $Q'\supseteq Q_0$ with $\mu(Q')>C$, which is possible since $\mu(\R^d) = \infty$. Define recursively $Q_{k+1}$ as the minimal dyadic cube such that $Q_k \subseteq Q_{k+1}$ and $\mu(Q_k) \leq \frac{1}{2}\mu(Q_{k+1})$.

Fix $n \in \N$ and let $e_1,\cdots,e_n \in E$ be pairwise disjoint (i.e. $\inf\{e_j,e_k\}=0$ for all $1\leq j,k \leq n$), such that $\nrm{e_1} \leq \cdots \leq \nrm{e_n}$. Define $\mc{D} = \bigcup_{k=0}^n Q_k$ and $f = \sum_{k=1}^n \ind_{Q_k\setminus Q_{k-1}} e_k$. Let $\mc{S} \subseteq \mc{D}$ be sparse such that
\begin{equation}\label{eq:sparsedom}
  \nrm{M_{\mc{D}}f}_E \leq C_{E,r} \,\mc{A}^\mu_{r,\mc{S}}(\nrm{f}_E).
\end{equation}
$\mu$-almost everywhere and let $x_0\in Q_0$ be such that \eqref{eq:sparsedom} holds.  Note that
\begin{equation*}
  \ip{\abs{f}}_{Q_k}^\mu \geq \frac{\mu(Q_k \setminus Q_{k-1})}{\mu(Q_k)} \abs{e_k}  \geq \has{1-\frac{1}{2}} \abs{e_k} =\frac{1}{2} \abs{e_k}.
\end{equation*}
By the elementary relations
\begin{align*}
  e+e'&=\sup\{e,e'\}+\inf\{e,e'\}\\
  \inf\{ \sup\{e,e'\},e''\}&=\sup\{\inf\{e,e''\},\inf\{e',e''\}\}
\end{align*} for $e,e',e'' \in E$, the disjoint vectors $e_k$'s satisfy
$
\sum_{k=1}^n e_k=\sup_{1 \leq k \leq n} e_k.
$ Therefore,
\begin{equation}\label{eq:sumestimate}
  \nrms{\sum_{k=1}^n e_k}_E \leq \nrmb{\sup_{1 \leq k \leq n} \abs{e_k}}_E \leq 2\nrmb{\widetilde{M}^\mu_{\mc{D}}f(x_0)}_E
\end{equation}
Moreover, since $\nrm{e_1} \leq \cdots \leq \nrm{e_n}$, we have that
\begin{equation*}
  \ip{\nrm{f}_E}_{Q_k}^\mu = \frac{1}{{\mu(Q_k)} }\sum_{j=1}^{k} \mu(Q_j\setminus Q_{j-1})\nrm{e_j}_E \leq  \nrm{e_k}_E.
\end{equation*}
which yields
\begin{equation}\label{eq:Aqestimate}
   \mc{A}^\mu_{r,\mc{S}} \hab{\nrm{f}_E}(x_0) \leq \has{\sum_{k=1}^n \hab{\ip{\nrm{f}}_{Q_k}^\mu}^{r}}^{\frac{1}{r}} \leq \has{\sum_{k=1}^n \nrm{e_k}_E^{r}}^\frac{1}{r},
\end{equation}
Combining \eqref{eq:sparsedom}, \eqref{eq:sumestimate}  and \eqref{eq:Aqestimate}, we deduce that
\begin{equation*}
  \nrms{\sum_{k=1}^n e_k}_E \leq C_{E,r} \has{\sum_{k=1}^n \nrm{e_k}_E^{r}}^\frac{1}{r},
\end{equation*}
for all pairwise disjoint vectors $e_1,\cdots,e_n \in E$ such that $\nrm{e_1} \leq \cdots \leq \nrm{e_n}$ and therefore for every collection of pairwise disjoint vectors in $E$.
This is called an \textit{upper $r$-estimate} for $E$. By \cite[Theorem 1.f.7]{LT79}, this implies that $E$ is $q$-convex for all $q \in [1,r)$.
\end{proof}

\section{Sparse domination for \texorpdfstring{$q$}{q}-convex lattices}\label{sec:sparse}

In this section we prove Theorem \ref{theorem:domination} from the introduction, which states the sufficiency of the $q$-convexity for the domination \eqref{eq:domination} to hold:
\begin{theoremnumber}{\ref*{theorem:domination}}[Sparse domination for lattice maximal operator]
Let $E$ be a Banach lattice and let $\mu$ be a locally finite Borel measure. Assume that $E$ has the Hardy--Littlewood property and is $q$-convex for some $q\in(1,\infty)$. Then for each finite collection $\mc{D}$ of dyadic cubes and for each locally integrable function $f:\R^d\to E$ there exists a sparse collection $\mc{S}\subseteq \mc{D}$ such that
\begin{equation*}
\nrmb{\widetilde{M}^\mu_{\mc{D}}f(x)}_E \leq C_{E,q} \,\mc{A}^\mu_{q,\mc{S}}(\nrm{f}_E)(x), \qquad \text{$\mu$-a.e. $x\in \R^d$.}
\end{equation*}
\end{theoremnumber}

  Cruz-Uribe, Martell, and P\'{e}rez \cite[Lemma 8.1]{CMP12} proved this domination in the case where $\mu$ is the Lebesgue measure and $E=\ell^q$, which is a prototypical Banach lattice that has the Hardy--Littlewood property and is $q$-convex. Their proof extends to the case of general measures and general Banach lattices as follows. First, in place of the estimate $0\leq \max\{a,b\}-b\leq a$ for all positive reals $a,b$, one uses the estimate
$$
0\leq \nrm{\sup\{e_1,e_2\}}_E^q -\nrm{e_2}_E^q\leq  \nrm{e_1}_E^q
$$
for all positive vectors $e_1,e_2$ in a $q$-convex lattice $E$. This estimate holds provided that the constant $C_{E,q}$ in the definition of $q$-convexity equals one, which can be arranged by passing to an equivalent norm \cite[Theorem 1.d.8]{LT79}. Second, in place of the usual Lerner median oscillation decomposition \cite{Le10}, one uses its variant for general measures \cite[Theorem 1.2]{Ha17b}.

We give an alternative proof for the sparse domination. Our proof is elementary in that it uses neither Lerner's median oscillation decomposition, unlike the Cruz-Uribe--Martell--P\'{e}rez  proof, nor renorming of the lattice.
Our proof is via the technique of stopping cubes, using a lattice-valued generalization of the Muckenhoupt--Wheeden {\it principal cubes} stopping condition. The generalized stopping condition has been applied to characterize lattice-valued two-weight norm inequalities \cite{Ha17} and is likely to have also other applications in the lattice-valued setting.

The generalized stopping condition is as follows. Let $f:\R^d\to E_+$ be a non-negative (in the lattice sense) locally integrable function. In the generalized stopping condition, we choose the maximal dyadic subcubes  $S'\subseteq S$ that satisfy the stopping condition $$\nrms{\sup_{\substack{Q \in \mc{D} \\S' \subseteq Q\subseteq S}}\ip{f}^\mu_{Q}}_E > 2 \nrmb{\widetilde{M}^\mu_{\mc{D}}}_{L^1(\mu;E) \to L^{1,\infty}(\mu;E)} \ip{\nrm{f}_E}_S^\mu.$$
Note that in the scalar-valued case $E_+=\R_+$ this reduces to choosing the  maximal dyadic subcubes $S'\subseteq S$ such that
$$
 \ip{f}^\mu_{S'} > 2 \ip{f}_S^\mu;$$
this is the Muckenhoupt--Wheeden {\it principal cubes} stopping condition, which originally appeared in \cite[Equation 2.5]{MW77}.

\begin{proof}[Proof of Theorem \ref{theorem:domination} via the technique of stopping cubes]
Let  $f\colon \R^d \to E$ be a locally integrable function, which may be taken positive without loss of generality. For a cube $S \in \mc{D}$, we define its stopping children $\ch_{\mc{S}}(S)$ to be the collection of maximal (w.r.t. set inclusion) cubes $S' \in \mc{D}$ such that $S' \subsetneq S$ and the cube $S'$ satisfies the stopping condition
\begin{equation}\label{eq:stopping}
  \nrms{\sup_{\substack{Q \in \mc{D} \\S' \subseteq Q\subseteq S}}\ip{f}^\mu_{Q}}_E > 2 \nrmb{\widetilde{M}^\mu_{\mc{D}}}_{L^1(\mu;E) \to L^{1,\infty}(\mu;E)} \ip{\nrm{f}_E}_S^\mu.
\end{equation}
Let $\mc{S}_0 := \{Q \in \mc{D}: Q \text{ maximal}\}$ and define recursively $\mc{S}_{k+1}:= \bigcup_{S \in \mc{S}_k} \ch_{\mc{S}}(S)$. We set $\mc{S}: = \bigcup_{k=0}^\infty \mc{S}_k$. For each $Q \in \mc{D}$, we define its stopping parent $\parent_{\mc{S}}(Q)$ as
\begin{equation*}\parent_{\mc{S}}(Q) \\
= \cbrace{S \in \mc{S}:S \text{ minimal (w.r.t. set inclusion) such that } Q \subseteq S}.
\end{equation*}

  First, we show that the collection $\mc{S}$ of dyadic cubes is sparse. Fix $S \in \mc{S}$ and let $E_{S}:=S\setminus \bigcup_{S'\in\ch_{\mc{S}}(S)}S'$. Define the set
  \begin{equation*}
    S^* := \cbraces{x \in \R^d: \nrmb{\widetilde{M}^\mu_{\mc{D}}\ha*{ f\ind_S}(x) }_E  > 2 \nrmb{\widetilde{M}^\mu_{\mc{D}}}_{L^1(E) \to L^{1,\infty}(E)} \ipb{\nrm{f}_E}_S^\mu}.
  \end{equation*}
Note that by the definition of the weak $L^1$-norm we have
\begin{equation}\label{eq:temp1}
\mu(S^*)\leq \frac{1}{2} \mu(S).
\end{equation}
   Moreover, for $S' \in \ch_{\mc{S}}(S)$ and $x \in S'$, we have
  \begin{equation*}
    \nrms{\widetilde{M}^\mu_{\mc{D}}\ha*{ f\ind_S}(x)}_E = \nrms{\sup_{Q \in \mc{D}} \ip{f\ind_S}_Q^\mu \ind_{S'}(x) }_E \geq \nrms{\sup_{\substack{Q \in \mc{D} \\ S' \subseteq Q\subseteq S }}\ip{f}_{Q}^\mu}_E
  \end{equation*}
  so $x \in S^*$ by \eqref{eq:stopping} and thus $S' \subseteq S^*$. Using the disjointness of $\ch_{\mc{S}}(S)$ and \eqref{eq:temp1},  we get
  \begin{equation*}
    \sum_{S' \in \ch_{\mc{S}}(S)} \mu(S') \leq \mu(S^*) \leq \frac{1}{2} \mu(S).
  \end{equation*}
  So $\mu(E_{S}) \geq \frac{1}{2}\mu(S)$, which means that $\mc{S}$ is a sparse collection of dyadic cubes.

  Next, we check the pointwise estimate.  Fix $S \in \mc{S}$, $x \in S$ and let $S_x \in \mc{D}$ be the minimal (w.r.t. set inclusion) cube such that $x \in S_x$ and $\parent_\mc{S}(S_x)=S$. By the minimality, we have
  $$
        \nrms{\sup_{\substack{Q \in \mc{D}\\\parent_{\mc{S}}(Q)=S}}\ip{f}_{Q}^\mu\ind_{Q}(x)}_E=\nrms{\sup_{\substack{Q \in \mc{D} \\S_x \subseteq Q\subseteq S}}\ip{f}_{Q}^\mu}_E \ind_S(x).
$$
  and by the condition  $\parent_\mc{S}(S_x)=S$, we have
  $$
 \nrms{\sup_{\substack{Q \in \mc{D} \\S_x \subseteq Q\subseteq S}}\ip{f}_{Q}^\mu}_E \ind_S(x) \leq  2 \nrmb{\widetilde{M}^\mu_{\mc{D}}}_{L^1(\mu,E) \to L^{1,\infty}(\mu,E)} \ip{\nrm{f}_E}^\mu_S \ind_S(x).
  $$
  Altogether,
\begin{equation}\label{eq:pointwise}
 \nrms{\sup_{\substack{Q \in \mc{D}\\\parent_{\mc{S}}(Q)=S}}\ip{f}_{Q}^\mu\ind_{Q}(x)}_E\leq 2 \nrmb{\widetilde{M}^\mu_{\mc{D}}}_{L^1(E) \to L^{1,\infty}(E)} \ip{\nrm{f}_E}_S^\mu \ind_S(x).
\end{equation}
Now, we have
  \begin{align*}
    \nrmb{\widetilde{M}^\mu_{\mc{D}}f(x)}_E&= \nrms{\sup_{S \in \mc{S}} \sup_{\substack{Q \in \mc{D}\\\parent_{\mc{S}}(Q)=S}}\ip{f}_{Q}^\mu\ind_{Q}(x)}_E\\
    &\leq  \nrms{\has{\sum_{S \in \mc{S}} \hab{\sup_{\substack{Q \in \mc{D}\\\parent_{\mc{S}}(Q)=S}}\ip{f}_{Q}^\mu\ind_{Q}(x)}^q}^\frac{1}{q}}_E \qquad && \nrm{\cdotroomy}_{\ell^\infty} \leq \nrm{\cdotroomy}_{\l^q}\\
    &\leq  C_{E,q} \,\has{\sum_{S \in \mc{S}} \nrms{\sup_{\substack{Q \in \mc{D}\\\parent_{\mc{S}}(Q)=S}}\ip{f}_{Q}^\mu\ind_{Q}(x)}_E^q}^\frac{1}{q} &&\text{$q$-convexity of $E$}\\
    &\leq  C_{E,q} \, \nrmb{\widetilde{M}^\mu_{\mc{D}}} \has{\sum_{S \in \mc{S}} \hab{\ip{\nrm{f}_E}_{S}^\mu}^q\ind_{S}(x)}^\frac{1}{q} &&\text{\eqref{eq:pointwise}},
  \end{align*}
with $\nrmb{\widetilde{M}^\mu_{\mc{D}}} := \nrmb{\widetilde{M}^\mu_{\mc{D}}}_{L^1(\mu;E) \to L^{1,\infty}(\mu;E)}$. By Proposition \ref{prop:calderonzygmund}, we have
$$\nrmb{\widetilde{M}^\mu_{\mc{D}}}_{L^1(\mu;E) \to L^{1,\infty}(\mu;E)}\leq C_p  \nrmb{\widetilde{M}^\mu_{\mc{D}}}_{L^p(\mu;E) \to L^{p}(\mu;E)}$$
for every $p \in (1,\infty)$. By the remark after Definition \ref{def:hardy_littlewood_property}, we have
$$
\nrmb{\widetilde{M}^\mu_{\mc{D}}}_{L^p(\mu;E) \to L^{p}(\mu;E)}\leq  \sup_{\mc{D}'} \nrmb{\widetilde{M}^{\dx}_{\mc{D}'}}_{L^p(\dx;E) \to L^{p}(\dx;E)}.
$$
Note that the quantity $\sup_{\mc{D}'} \nrmb{\widetilde{M}_{\mc{D}'}^{\dx}}_{L^p(\dx;E) \to L^{p}(\dx;E)}$ is finite for some $p \in (1,\infty)$ by the assumption that $E$ has the Hardy-Littlewood property. This completes the proof of the theorem.
\end{proof}

\section{Weighted estimates for non-dyadic maximal functions}\label{sec:weighted_estimates}
As well-known, via the domination of an operator by sparse operators, the weighted bounds for sparse operator carry over to the dominated operator. In this section, we mention weighted bounds that carry over  via the domination from sparse operators to the  non-dyadic lattice Hardy--Littlewood maximal operator.

\subsection*{Non-dyadic lattice Hardy--Littlewood maximal operator}

We define the {\it non-dyadic lattice Hardy--Littlewood maximal operator} $\widetilde{M}^\mu$ as follows: for a locally integrable function $f\colon \R^d \to E$, we set
\begin{equation}\label{def:nondyadicmaximal}
\widetilde{M}^\mu f(x):=\sup_{Q} \ip{\abs{f}}_Q^\mu  \ind_Q(x),\qquad x \in \R^d,
\end{equation}
where the supremum is taken in the lattice sense over all cubes $Q \subseteq \R^d$ with sides parallel to the coordinate axes.

For this definition to make sense, the supremum needs to exist for $\mu$-a.e. $x \in \R^d$, and $\widetilde{M}^\mu f$ needs to be strongly $\mu$-measurable, i.e. it needs to be pointwise approximable by simple functions (see \cite[Chapter 1]{HNVW16} for more on strong measurability). This is the case if the Banach lattice is {\it order continuous}. (On order continuity, see for example \cite[Section 1.a]{LT79}.) Since, in particular, every reflexive Banach lattice is order continuous, this a rather general sufficient condition.

\begin{lemma}[Well-definedness of the non-dyadic lattice maximal operator]\label{lemma:welldefmaximal}Let $E$ be an order continuous Banach lattice and $\mu$ be a locally finite Borel measure.  Then for every simple function $f:\R^d \to E$ the maximal function $\widetilde{M}^\mu f$ exists and is strongly $\mu$-measurable.
\end{lemma}

\begin{proof}
Note that since $E$ is order-continuous, the space of all strongly $\mu$-measurable functions $L^0(\mu;E)$ is order-complete by \cite[Theorem 2.6]{Gr82}, i.e. every order bounded set in $L^0(\mu;E)$ has a supremum in $L^0(\mu;E)$.

Let $f\colon\R^d \to E$ be a simple function, that is, $f = \sum_{k=1}^n e_k \ind_{A_k}$ with $e_1,\cdots,e_n \in E$ and $A_1,\cdots,A_n \subseteq \R^d$ measurable, pairwise disjoint and $\mu(A_k)<\infty$ for $k=1,\cdots,n$.  Since we have for all cubes $Q \subseteq \R^d$ that
\begin{equation*}
  \ip{\abs{f}}_Q^\mu \ind_{Q} \leq \has{\sum_{k=1}^n \abs{e_k}} \ind_{\R^d} \in L^0(\mu;E),
\end{equation*}
it follows that
\begin{equation*}
\widetilde{M}^\mu f=\sup_{Q} \ip{\abs{f}}_Q^\mu  \ind_Q \in L^0(\mu;E)\qedhere
\end{equation*}
\end{proof}

\subsection*{Muckenhoupt weights} We now turn to the weighted estimates for the non-dyadic Hardy--Littlewood maximal operator. For this, we fix $\mu$ to be the Lebesgue measure $\dx$ and denote $\widetilde{M}:=\widetilde{M}^{\dx}$, $\widetilde{M}_{\mc{D}}:=\widetilde{M}^{\dx}_{\mc{D}}$, $\mc{A}_{q,\mc{S}}:= \mc{A}_{q,\mc{S}}^{\dx}$ and $\ip{\cdotroomy}_Q:= \ip{\cdotroomy}^{\dx}_Q$.

A \emph{weight} is a nonnegative locally integrable function $w \colon \R^d \to (0,\infty)$. For $p \in [1,\infty)$, the weighted Lebesgue--Bochner space $L^p(w;E)$ is the space of all $f \in L^0(\dx;E)$ such that
\begin{equation*}
  \nrm{f}_{L^p(w;E)}:= \has{\int_{\R^d} \nrm{f}^p_E w  \dx}^{1/p} < \infty.
\end{equation*}
For $p \in [1,\infty)$, the class of the \emph{Muckenhoupt $A_p$-weights} contains all weights $w$ such that
\begin{equation*}
  [w]_{A_p} := \sup_Q \ip{w}_Q \ipb{w^{-\frac{1}{p-1}}}_Q^{p-1} < \infty,
\end{equation*}
where the supremum is taken over all cubes $Q \subseteq \R^d$ with sides parallel to the coordinate axes, and where the second factor is replaced by $\nrm{w^{-1}}_{L^\infty(Q)}$ for $p=1$. For $p=\infty$, the class contains all weights such that
\begin{equation*}
  [w]_{A_\infty} = \frac{\int_Q M(w \ind_Q)  \dd x }{\int_Qw  \dd x } < \infty,
\end{equation*}
 where $M$ is the usual (scalar) Hardy-Littlewood maximal operator. We call $[w]_{A_p}$ the {\it $A_p$-characteristic} of $w$. For a general overview of Muckenhoupt weights, see \cite[Chapter 9]{Gr09}, and for an introduction to the $A_\infty$-characteristic, see \cite{HPR12} and the references therein.

\subsection*{Weighted bounds for maximal operators}As well-known, there are boundedly many shifted dyadic systems such that every cube is contained in some dyadic cube of comparable side length (see for example \cite[Lemma 3.2.26]{HNVW16}). Hence, as well-known, non-dyadic maximal operators can be dominated by dyadic maximal operators.
Via the domination of non-dyadic lattice maximal operators by dyadic lattice maximal operators and the domination of dyadic lattice maximal operators by sparse operators, the weighted bounds for sparse operator carry over to the non-dyadic lattice maximal operator.
In this way the weighted bounds for sparse operators from
\begin{itemize}
\item \cite[Theorem 1.1. and Theorem 1.2.]{HL18} in the case $L^p(w)\to L^p(w)$ and $L^p(w)\to L^{p,\infty}(w)$
\item \cite[Theorem 1.3]{FN17} in the case $L^1(w)\to L^{1,\infty}(w)$
\end{itemize} yield the following weighted estimates:

\begin{corollary}\label{corollary:weightedestimate}
  Let $E$ be an order-continuous Banach lattice. Assume that $E$ has the Hardy--Littlewood property and is thus $q$-convex for some $q\in(1,\infty)$. Then for all $p \in (1,\infty)$, $w \in A_p$ and $f \in L^p(w;E)$ we have
  \begin{align}
    \nrmb{\widetilde{M}f}_{L^p(w;E)} &\leq C_{E,p,q,d} \, \label{eq:strongApAinfty} [w]_{A_p}^\frac{1}{p}\hab{[w]_{A_\infty}^{\frac{1}{q}-\frac{1}{p}}+[w^{1-p'}]_{A_\infty}^{\frac{1}{p}}} \nrmb{f}_{L^p(w;E)} \\ &\leq \, C_{E,p,q,d}[w]_{A_p}^{\max\cbraceb{\frac{1}{p-1},\frac{1}{q}}}  \nrmb{f}_{L^p(w;E)} \label{eq:strongAp},
    \intertext{and if $p \neq q$ we have}
    \nrmb{\widetilde{M}f}_{L^{p,\infty}(w;E)} &\leq \label{eq:weakApAinfty} C_{E,p,q,d}\,[w]_{A_p}^{\frac{1}{p}}\hab{[w]_{A_\infty}^{\frac{1}{q}-\frac{1}{p}}+1}\nrmb{f}_{L^p(w;E)}  \\&\leq C_{E,p,q,d}\,[w]_{A_p}^{\max\cbraceb{\frac{1}{p},\frac{1}{q}}} \nrmb{f}_{L^p(w;E)}. \label{eq:weakAp}
  \intertext{If $w \in A_1$ and $f \in L^1(w;E)$ we have}\label{eq:weakA1}
    \nrmb{\widetilde{M}f}_{ L^{1,\infty}(w;E)} &\leq C_{E,d}\,[w]_{A_1}\hab{1+ \log([w]_{A_\infty})} \nrm{f}_{L^{1}(w;E)}
  \end{align}
\end{corollary}
In the particular case $E=\ell^q$, the strong-type weighted bound \eqref{eq:strongAp} together with its sharpness was proved in \cite{CMP12}. After the appearance of this manuscript on arXiv, another manuscript appeared, in which the weighted bounds \eqref{eq:strongApAinfty} and \eqref{eq:weakApAinfty} for the lattice maximal operator were deduced independently in the particular case $E=\ell^q$, see \cite[Theorem 2]{CLPR17}.

\begin{remark}
 In the particular case $E=\ell^q$, the dependence on the $A_p$-characteristic is sharp both in the strong-type weighted estimate \eqref{eq:strongAp} (see \cite{CMP12}) and  in the weak-type weighted estimate  \eqref{eq:weakAp} (this follows from combining  \cite{CMP12} and \cite[Theorem 1]{PR17}).
In the general case that $E$ is Banach lattice that is $q$-convex for some $q\in(1,\infty)$, the exponent
  $$
  q^*:=\sup\{q\in(1,\infty) : \text{$E$ is $q$-convex}\}
  $$
  is critical: The strong-type weighted estimate \eqref{eq:strongAp} with the dependence $$[w]_{A_p}^{\max\cbraceb{\frac{1}{p-1},\frac{1}{q}}}$$ holds for all $q<q^*$ and fails for all $q>q^*$. Similarly, the weak-type \eqref{eq:weakAp} weighted estimate with the dependence $$[w]_{A_p}^{\max\cbraceb{\frac{1}{p},\frac{1}{q}}}$$ holds for all $q<q^*$ and fails for all $q>q^*$.
 This  follows from embedding a copy of $\ell^{q}_n$ with $q<q^*$ into the lattice $E$ for a large enough $n$ (by applying \cite[Theorem 1.f.12]{LT79}) and using the sharpness in the case $\ell^q_n$.
This sharpness for weighted estimates can be compared with the sharpness for domination, see Corollary \ref{corollary:sharpness}.
\end{remark}

\appendix

\section{Strong \texorpdfstring{$L^p$}{Lp}-bound  implies weak \texorpdfstring{$L^1$}{L1}-bound}\label{sec:lpl1}
As well-known, for the dyadic lattice Hardy Littlewood maximal operator the strong $L^p$-boundedness implies the weak $L^1$-boundedness. This result can be proven by viewing the lattice maximal operator as a vector-valued singular integral operator (see \cite{GMT93, GMT98}) and using the Calder\'on--Zygmund decomposition, or alternatively, by viewing the lattice maximal operator as a martingale transform (see \cite{MT00}) and using the Gundy decomposition. In this Appendix, we give an elementary proof of this result.

\begin{proposition}\label{prop:calderonzygmund}
  Let $E$ be a Banach lattice, $\mu$ a locally finite Borel measure, and $\mc{D}$ a finite collection of dyadic cubes. Then for all $p \in (1,\infty)$
   \begin{equation*}
     \nrmb{\widetilde{M}^\mu_{\mc{D}}}_{L^{1,\infty}(\mu;E)\to L^1(\mu;E)} \leq C_p  \nrmb{\widetilde{M}^\mu_{\mc{D}}}_{L^{p}(\mu;E)\to L^p(\mu;E)} .
   \end{equation*}
\end{proposition}

\begin{proof}
    Fix $f \in L^1(\mu;E)$, which may be taken positive without loss of generality.
    Let $\widetilde{\mc{D}}$ be the dyadic grid such that $\mc{D} \subseteq \widetilde{\mc{D}}$ and for a cube $Q \in \widetilde{\mc{D}}$ let its dyadic parent $\hat{Q}$ be the minimal cube $Q' \in \widetilde{\mc{D}}$ such that $Q \subsetneq Q'$. Define for  $\lambda >0$
    \begin{equation*}
      \mc{S} := \cbrace{Q \in \widetilde{\mc{D}} \text{ maximal with } \ip{\nrm{f}_E}_Q> \lambda}.
    \end{equation*}
    We write $\Omega := \bigcup_{S \in \mc{S}} S$. For a fixed cube $Q \in \mc{D}$ we have
  \begin{align*}
    \ip{ f}_Q &= \sum_{\substack{S\in \mc{S}\\ S \subsetneq Q}} \ip{f\ind_{S}}_Q &&+ &&&&\sum_{\substack{S\in \mc{S}\\ S \supseteq Q}}\ip{f\ind_{S}}_Q &&&+ &&&&&\ip{f\ind_{\Omega^\comp}}_Q\\
    &\leq  \sum_{S\in \mc{S}} \frac{\mu\ha{S}} {\mu\ha{\hat{S}}} \ipb{\ip{f}_S\ind_{\hat{S}}}_Q &&+ &&&&\sum_{\substack{S\in \mc{S}\\ S \supseteq Q}}\ip{f}_Q &&&+ &&&&&\ip{f \ind_{\Omega^\comp}}_Q,
  \end{align*}
  as $\hat{S} \subseteq Q$ if $S \subsetneq Q$. Therefore, we have the decomposition
  \begin{equation}\label{eq:ourdecomposition}  \begin{split} \widetilde{M}^\mu_{\mc{D}} f &\leq \widetilde{M}^\mu_{\mc{D}}\has{\sum_{S \in \mc{S}} \frac{\mu\ha{S}}{\mu\ha{\hat{S}}} \ip{f}_S\ind_{\hat{S}}  + f \ind_{\Omega^\comp}} + \sup_{Q \in \mc{D}} \sum_{\substack{S\in \mc{S}\\ S \supseteq Q}} \ip{f}_Q \ind_Q
    \\ &=:\widetilde{M}^\mu_{\mc{D}}(g_1+g_2)+b.
  \end{split}
  \end{equation}

  Note that $b$ is supported on $\Omega$ and $\Omega=\{M_{\widetilde{\mc{D}}}(\nrm{f}_E) > \lambda\}$, where  $M_{\widetilde{\mc{D}}}$ is the usual dyadic (scalar) Hardy--Littlewood maximal operator over the dyadic grid $\widetilde{\mc{D}}$. By the weak $L^1$-boundedness of $M_{\widetilde{\mc{D}}}$ (see for example \cite{St93}), we have
  \begin{equation}\label{eq:bweakL1}
  \mu\hab{\nrm{b}_E>\lambda} \leq  \mu\hab{M_{\widetilde{\mc{D}}}(\nrm{f}_E) > \lambda} \leq \frac{1}{\lambda} \nrm{f}_{L^1(\mu;E)}.
  \end{equation}
 Since $\mc{S}$ is a family of disjoint dyadic cubes, we have by \cite[Lemma 3.3]{LMP14} that
  \begin{equation}\label{eq:g1bound}
  \begin{aligned}
    \nrm{g_1}_{L^p(E)}^p &\leq
    \int_{\R^d} \has{\sum_{S \in \mc{S}}  \frac{\mu\ha{S}}{\mu\ha{\hat{S}}}\ip{\nrm{f}_E}^\mu_S \ind_{\hat{S}} }^p  \dd x\\
    &\leq C_p
    \has{\sup_{S \in \mc{S}} \ip{\nrm{f}_E}^\mu_{\hat{S}}}^{p-1} \int_\Omega \nrm{f}_E \dd x \leq C_p \, \lambda^{p-1} \nrm{f}_{L^1(E)}.
    \end{aligned}
  \end{equation}
  By the Lebesgue differentiation theorem and the definition of $\Omega$, we have
  \begin{equation*}
    \nrm{g_2(x)}_E = \nrm{f(x)}_E \ind_{\Omega^c}(x) \leq \sup_{Q \in \widetilde{\mc{D}}: Q \subseteq \Omega^c}\ip{\nrm{f}_E}_Q  \leq \lambda
  \end{equation*}
  for $\mu$-a.e. $x \in \R^d$ and therefore
  \begin{equation}\label{eq:g2bound}
    \nrm{g_2}_{L^p(E)}^p \leq \lambda^{p-1} \nrm{f}_{L^1(\mu;E)}.
  \end{equation}
  Combining \eqref{eq:bweakL1}, \eqref{eq:g1bound} and \eqref{eq:g2bound} we obtain
  \begin{align*}
    \mu\ha*{\nrmb{\widetilde{M}^\mu_{\mc{D}}f}_E > 2\lambda} &\leq \mu\has{\nrmb{\widetilde{M}^\mu_{\mc{D}}(g_1+g_2)}_E > \lambda}+\mu\hab{\nrm{b}_E>\lambda} \\
    &\leq \nrmb{\widetilde{M}^\mu_{\mc{D}}f}_{L^{p,\infty}(\mu;E)\to L^{p,\infty}(\mu;E)}  \cdot \frac{\nrm{g_1+g_2}^p_{L^p(\mu;E)}}{\lambda^p} + \frac{1}{\lambda} \nrm{f}_{L^1(\mu;E)}\\
    &\leq C_p\frac{1}{\lambda} \nrmb{\widetilde{M}^\mu_{\mc{D}}f}_{L^{p,\infty}(\mu;E)\to L^{p,\infty}(\mu;E)}   \nrm{f}_{L^1(\mu;E)}\\
    &\leq C_p \frac{1}{\lambda} \nrmb{\widetilde{M}^\mu_{\mc{D}}f}_{L^{p}(\mu;E)\to L^{p}(\mu;E)}   \nrm{f}_{L^1(E)},
  \end{align*}
which completes the proof of the proposition.
\end{proof}

\begin{remark}
  The functions $g_1$ and $g_2$ are a subpart of the good part of the non-doubling Calder\'on--Zygmund decomposition \cite[Theorem 2.1]{LMP14}. Our decomposition \eqref{eq:ourdecomposition} can be viewed as a hands-on variant of that Calder\'on–-Zygmund decomposition.
\end{remark}

\bibliographystyle{plain}
\bibliography{sparsebib}

\end{document}